\documentclass[12pt,reqno]{amsart}

\usepackage[margin=1in]{geometry}
\usepackage{amsmath,amssymb,amsthm}
\usepackage{mathtools}
\usepackage{microtype}
\usepackage{enumitem}
\usepackage{hyperref}
\usepackage[nameinlink,capitalize]{cleveref}
\usepackage{tikz}

\hypersetup{
    colorlinks=true,
    linkcolor=blue,
    citecolor=blue,
    urlcolor=blue
}

\numberwithin{equation}{section}

\newtheorem{theorem}{Theorem}[section]

\newtheorem{lemma}[theorem]{Lemma}
\newtheorem{corollary}[theorem]{Corollary}
\newtheorem{remark}[theorem]{Remark}
\newtheorem{conjecture}{Conjecture}
\newtheorem{definition}{Definition}
\newtheorem{openproblem}{Open Problem}
\newtheorem{example}{Example}

\newcommand{\T}{\mathbb T}
\newcommand{\Z}{\mathbb Z}

\newcommand{\R}{\mathbb R}
\newcommand{\A}{\mathcal A}
\newcommand{\M}{\mathcal M}

\newcommand{\AM}{\mathcal A_M}
\newcommand{\esssup}{\operatorname*{ess\,sup}}

\newcommand{\AC}{{\rm AC\,}}

\title{$L^\infty$ Variational Approximation of the Aubry Set}

\author[H. V. Tran, Y. Yu]{Hung V. Tran, Yifeng Yu}

\date{}

\thanks{
H. V. Tran is partially supported by NSF grant DMS-2348305. 
}

\address[H. V. Tran]
{
Department of Mathematics, 
University of Wisconsin-Madison, Van Vleck Hall, 480 Lincoln Drive, Madison, Wisconsin 53706, USA}
\email{hung@math.wisc.edu}

\address[Y. Yu]
{
Department of Mathematics, 
University of California at Irvine, 
California 92697, USA}
\email{yifengy@uci.edu}

\keywords{Aubry set; $L^\infty$ variational problem; viscosity solutions; weak KAM theory}

\subjclass[2020]{35B10, 35B40, 35F21, 49L25}

\begin{document}

\begin{abstract}
Let $H\in C^\infty(\mathbb R^n\times\mathbb T^n)$ be a periodic Tonelli Hamiltonian with critical value $c$.
For each $k\in\mathbb N$, let $u_k$ be the normalized minimizer of the variational functional introduced by Evans \cite{Evans03},
\[
    I_k[w]=\int_{\mathbb T^n} e^{kH(Dw,x)}\,dx,
    \qquad \int_{\mathbb T^n}w\,dx=0.
\]
If $u_\infty$ is a uniform limit of a subsequence of $\{u_k\}$ and the Mather quotient $(\mathcal A_M,\delta_M)$ satisfies
$\mathcal H^1(\mathcal A_M,\delta_M)=0$, then $u_\infty$ is a critical subsolution that is strict outside $\A$ and 
\[
    \mathcal A
    =
    \{x\in\mathbb T^n\,:\,Du_\infty(x)\ \text{exists and }H(Du_\infty(x),x)=c\}=\{x\in\mathbb T^n\,:\,u_\infty(x)=u_-(x)\},
\]
where $\mathcal A$ is the projected Aubry set and $u_-$ is the backward weak KAM solution associated with $u_\infty$. 
In particular, by the theorem of Fathi--Figalli--Rifford \cite{FFR09}, this conclusion holds for all smooth Tonelli Hamiltonians on $\mathbb T^n$ when $n\leq3$. 
This characterization also suggests  a natural numerical localization principle for approximating the entire Aubry set through near-contact sets between \(u_k\) and its large-time backward Lax--Oleinik evolution.
\end{abstract}

\maketitle

\section{Introduction}

Evans in \cite{Evans03} introduced an approximate variational principle for weak KAM theory. 
Precisely speaking, consider a periodic Tonelli Hamiltonian \(H=H(p,x)\in C^\infty(\R^n\times \T^n)\). 
Minimize
\begin{equation}\label{eq:Ik}
    I_k[w]
    :=
    \int_{\T^n} e^{kH(Dw,x)}\,dx \qquad \text{subject to} \qquad \int_{\T^n}w\,dx=0
\end{equation}
and then let \(k\to\infty\).  
The exponential is a soft maximum:
\[
    \frac1k\log I_k[w]
    \longrightarrow
    \esssup_{\T^n}H(Dw,x).
\]
The approximation is therefore naturally related both to weak KAM theory and to the \(L^\infty\) variational problems.  See \cref{sec:pre} for  background and definitions on these two subjects and their connections. We also would like to mention that the same approximation has an exact interpretation as a stationary first-order mean-field game with logarithmic coupling \cite{GY20}.

Let \(u_k\) minimize \eqref{eq:Ik} and suppose that \(u_k\to u_\infty\) uniformly, up to a subsequence if necessary. 
It was proved in \cite{Evans03} that $u_\infty\in W^{1,\infty}(\T^n)$ is a subsolution to 
\begin{equation}\label{eq:cell-sub}
H(Du_\infty,x)=c \qquad \text{in $\mathbb T^n$}.
\end{equation}
Here $c$ is the critical value associated to $H$. 
Moreover, $u_\infty$ is an absolute minimizer for $H$ and is a viscosity solution to the following Aronsson equation
\begin{equation}\label{eq:Aronsson}
A_H(u)=H_{p_i}(Du,x)H_{p_j}(Du,x)u_{x_ix_j}+H_{x_i}(Du,x)H_{p_i}(Du,x)=0.
\end{equation}

For convenience, following the terminology of \cite{FathiSiconolfi04},
we call any subsolution of \eqref{eq:cell-sub} a \emph{critical subsolution}. For any critical subsolution,  the projected Aubry set \(\A\) is the obstacle to strict inequality (see Section \ref{sec:pre}).  
Accordingly,  it is natural to think that Evans' variational algorithm might provide an approach to detect \(\A\). 
For any critical subsolution,  write

\begin{equation}\label{eq:Du}
    D_u
    :=
    \bigl\{
       x\in\T^n\,:\,
       Du(x)\ \text{exists and }H(Du(x),x)=c
    \bigr\}.
\end{equation}
    If $u$ is a periodic absolute minimizer for $H$, then owing to \cite[Theorem 3.4]{Yu07},
\[
D_u= \{x\in\mathbb T^n\,:\,u(x)=u_-(x)\}.
\]
Here $u_-$ is the backward weak KAM solution associated with $u$.

This led to the following conjecture in \cite{Yu07}.

\begin{conjecture}[Conjecture 3.9 in \cite{Yu07}]\label{conj:original}

\[
    D_{u_\infty}=\A.
\]
\end{conjecture}

The one-dimensional case was already known in \cite{Yu07}; related one-dimensional selection questions for the associated Mather measures were studied further in \cite{GISMY10}.

The purpose of this paper is to prove the above Conjecture \ref{conj:original} under a geometric smallness assumption on the Mather quotient.  
Let \(h(x,y)\) be the Peierls barrier and define the Mather semidistance on \(\A\) by
\[
    \delta_M(x,y):=h(x,y)+h(y,x).
\]
The Mather quotient \(\AM\) is obtained by identifying \(x,y\in\A\) whenever \(\delta_M(x,y)=0\). 
Denote by $\mathcal H^1(\AM,\delta_M)$ the one-dimensional Hausdorff measure of the Mather quotient $\AM$ with respect to the metric $\delta_M$.

Our main result is the following.

\begin{theorem}\label{thm:main}
Assume that \(H\in C^\infty(\R^n\times\T^n)\) is a periodic Tonelli Hamiltonian.  Let \(u_k\) be the normalized minimizers of \eqref{eq:Ik}, and let \(u_\infty\) be a uniform subsequential limit. 
If
\begin{equation}\label{eq:H1zero}
    \mathcal H^1(\AM,\delta_M)=0,
\end{equation}
then $u_\infty$ is a critical subsolution that is strict outside $\A$ and
\[
    \{x\in\mathbb T^n\,:\,u_\infty(x)=u_-(x)\}=D_{u_\infty}=\A.
\]
Here $u_-(x)$ is given by (\ref{eq:backward}) with $u=u_\infty$.\end{theorem}

The dimension enters only through \eqref{eq:H1zero}.  
A theorem of Fathi--Figalli--Rifford \cite{FFR09} therefore immediately gives the following consequence.

\begin{corollary}\label{cor:lowdim}
Let \(H\in C^\infty(\R^n\times\T^n)\) be a periodic Tonelli Hamiltonian. 
If \(n\leq3\), then every uniform subsequential limit \(u_\infty\) of Evans' exponential minimizers  is a critical subsolution that is strict outside $\A$ and satisfies
\[
     \{x\in\mathbb T^n\,:\,u_\infty(x)=u_-(x)\}=D_{u_\infty}=\A.
\]
In particular, Conjecture~\ref{conj:original} holds for all smooth Tonelli Hamiltonians in dimensions \(n\leq3\).
\end{corollary}

Conjecture \ref{conj:original} was formulated in \cite{Yu07}, where the connection between Evans’ variational approximation and weak KAM theory was developed. 
The geometric result of Fathi--Figalli--Rifford \cite{FFR09} implies that $\mathcal H^1(\AM,\delta_M)=0$ for smooth Tonelli Hamiltonians in dimensions at most three. 
\cref{thm:main} and \cref{cor:lowdim} therefore resolve the conjecture in precisely the dimensions in which this geometric smallness is known unconditionally.

\begin{remark}
Although the identity $D_{u_\infty}=\A$ is conceptually clean, the characterization
\[
\{x\in\mathbb T^n\,:\,u_\infty(x)=u_-(x)\}
\]
provides a natural numerical localization scheme for the Aubry set and is more convenient for computation. 

The numerical recovery of the full Aubry set has been a longstanding challenge. 
Although several numerical procedures for detecting the Aubry set have been proposed, rigorous convergence of approximation schemes to the entire Aubry set is considerably more delicate and often requires additional dynamical assumptions. 
Recently, Camilli and Mendico \cite{CamilliMendico26} studied semi-discrete approximations of Aubry and
Mather sets via the discrete Lax--Oleinik semigroup. 
They obtained upper Kuratowski convergence in general and full convergence of the Aubry set under a hyperbolicity assumption.
We also refer the reader to \cite{NS12, ST18}.

 Recall that the backward Lax--Oleinik semigroup (\ref{eq:Lax}) is nonexpansive in the uniform norm:
\[
\|T_t^-u-T_t^-v\|_{L^\infty(\mathbb T^n)}
\leq
\|u-v\|_{L^\infty(\mathbb T^n)}
\qquad\text{for all }t\geq 0.
\]
Consequently, if $u_{k_j}\to u_\infty$ uniformly and
\[
u_{k_j,-}=\lim_{t\to\infty}(T_t^-u_{k_j}+ct),
\qquad
u_-=\lim_{t\to\infty}(T_t^-u_\infty+ct),
\]
then
\[
\|u_{k_j,-}-u_-\|_{L^\infty(\mathbb T^n)}
\leq
\|u_{k_j}-u_\infty\|_{L^\infty(\mathbb T^n)}
\longrightarrow 0.
\]
Hence
\[
|u_{k_j}-u_{k_j,-}|
\longrightarrow
|u_\infty-u_-|
\qquad\text{uniformly on }\mathbb T^n.
\]
Since this holds for every convergent subsequence, and since for \(n\leq3\) our theorem gives
\[
\mathcal A=\{x\in\mathbb T^n:u_\infty(x)=u_-(x)\},
\]
if
\[
\mathcal A_{k,\varepsilon}
:=
\{x\in\mathbb T^n:
|u_k(x)-u_{k,-}(x)|\leq\varepsilon\},
\]
then
\[
\lim_{\varepsilon\downarrow0}
\limsup_{k\to\infty}
d_H(\mathcal A_{k,\varepsilon},\mathcal A)=0.
\]
Here $d_H$ refers to the Hausdorff distance between two sets. 
\end{remark}
The two main computational ingredients in this procedure--the convex minimization problem defining \(u_k\) and the numerical approximation of the backward Lax--Oleinik semigroup, equivalently the solution of a Cauchy problem for a convex Hamilton--Jacobi equation--can be handled by standard numerical methods. 
Developing efficient implementations, understanding the interaction between the discretization parameters and the limits \(k,t\to\infty\), \(\epsilon\to 0\) and testing the procedure on more complicated nonintegrable examples are left for future work.

\begin{remark}
The exponential form of Evans' approximation is not essential for conclusions in this paper. 
For instance, after choosing \(C\) so that \(C+H>0\), one may consider
\[
I_k^C[u]
=
\int_{\mathbb T^n}
\bigl(C+H(Du,x)\bigr)^k\,dx.
\]
The proof could be carried out similarly and yield the corresponding conclusions for subsequential limits. 
The exponential approximation has the advantage of avoiding an additive shift and leads to a particularly natural smooth variational problem and associated probability measures.
\end{remark}

\section{Preliminaries}\label{sec:pre}
In this section, we will briefly review basic facts in weak KAM theory,  $L^\infty$ variational problems, and their connections. 
We refer the reader to \cite{Fathi08, TranHJ, TYbook, Yu07} for more details. 

\subsection{Cell problem and the critical value} 
Let $\mathbb T^n=\R^n/ \Z^n$ be the $n$-dimensional flat torus.  
Suppose that  \(H\in C(\R^n\times\T^n)\) is periodic in \(x\) and coercive in $p$, i.e.,
\[
\lim_{|p|\to \infty}\min_{x\in \mathbb T^n}H(p,x)=\infty.
\]
Then, by \cite{LPV}, there exists a unique constant \(c\in \R\) such that the following equation, usually referred to as the \emph{cell problem}, admits a periodic viscosity solution:
\begin{equation}\label{eq:cell}
H(Du,x)=c
\qquad \text{on }\mathbb T^n.
\end{equation}
See also \cite{Fathi08} for another proof in the case where \(H\) is convex and superlinear in the momentum variable. 
In this setting, solutions of \eqref{eq:cell} are referred to as weak KAM solutions in \cite{Fathi08}.

More generally, for a given \(P\in\mathbb R^n\), replacing \(H(p,x)\) by \(H(P+p,x)\) yields a corresponding critical value \(c=c(P)\). 
As a function of \(P\), this is called the \emph{effective Hamiltonian}
\[
\overline H(P)
\]
in \cite{LPV} and in homogenization theory, and corresponds to the \emph{Mather alpha function}
\[
\alpha(P)
\]
in the Aubry--Mather and weak KAM theories.

\subsection{Hamiltonian and Lagrangian setting}
Hereafter, we consider a Tonelli Hamiltonian.  
For convenience, we assume that it is smooth:
\[
    H\in C^\infty(\R^n\times\T^n)
\]
satisfying:
\begin{enumerate}[label=\textup{(H\arabic*)}]
    \item\label{H1}
    \(H(\cdot,x)\) is strictly convex in the Tonelli sense for every \(x\in\T^n\), equivalently,
    \(D^2_{pp}H(p,x)\) is positive definite;
    \item\label{H2}
    \(H\) is uniformly superlinear:
    \[
        \lim_{|p|\to\infty}\frac{H(p,x)}{|p|}
        =
        +\infty
        \qquad\text{uniformly in }x\in\T^n.
    \]
\end{enumerate}

The corresponding Lagrangian is
\[
    L(q,x)
    :=
    \sup_{p\in\R^n}\{p\cdot q-H(p,x)\}.
\]
An Euler-Lagrange flow $(\dot \xi, \xi):\R\to \R^{2n}$ satisfies that
\begin{equation}\label{eq:EL-flow}
\frac{d}{dt}(D_qL(\dot \xi(t),\xi(t)))=D_xL(\dot \xi(t),\xi(t)) \qquad \text{for $t\in \R$}.
\end{equation}
The following  Fenchel inequality is frequently used
\begin{equation}\label{eq:Fenchel}
    p\cdot q\leq H(p,x)+L(q,x).
\end{equation}

A periodic Lipschitz function \(w\) is called a \emph{critical subsolution} if
\[
    H(Dw,x)\leq c
    \qquad\text{a.e. in }\T^n.
\]
By convexity, this is equivalent to being a viscosity subsolution of \eqref{eq:cell}. 
Note that if \(w\) is a critical subsolution and
\(\gamma:[a,b]\to\mathbb R^n\) is absolutely continuous, then (\ref{eq:Fenchel}) implies that
\begin{equation}\label{eq:dominance}
w(\gamma(b))-w(\gamma(a))
\leq
\int_a^b \bigl(L(\dot\gamma(t),\gamma(t))+c\bigr)\,dt.
\end{equation}
In the terminology of \cite{Fathi08}, this means that \(w\) is
\emph{dominated by \(L+c\)}. If the equality holds
\[
w(\gamma(b))-w(\gamma(a))
=
\int_a^b \bigl(L(\dot\gamma(t),\gamma(t))+c\bigr)\,dt,
\]
then \(\gamma\) is said to be \((w,L,c)\)-calibrated; see
\cite{Fathi08}. A calibrated curve minimizes the Lagrangian action between
any two of its points and therefore satisfies the Euler--Lagrange equation
\eqref{eq:EL-flow}.

In addition, a critical subsolution \(w\) is called \emph{ strict outside $\A$} if, for every
closed set \(K\subset \mathbb T^n\) satisfying
\[
K\cap\mathcal A=\emptyset,
\]
there exists \(\eta_K>0\) such that
\[
H(Dw,x)\leq c-\eta_K
\qquad\text{for a.e. }x\in K.
\]

For \(g\in C(\mathbb T^n)\), we denote by \(T_t^-\) the backward Lax--Oleinik semigroup defined by
\begin{equation}\label{eq:Lax}
T_t^-g(x)
:=
\inf_{\substack{\gamma\in \AC([0,t];\mathbb T^n)\\ \gamma(t)=x}}
\left\{
g(\gamma(0))
+
\int_0^t L(\dot \gamma(s),\gamma(s))\,ds
\right\}.
\end{equation}
Equivalently, \(w(x,t):=T_t^-g(x)\) is the unique viscosity solution of the Cauchy problem
\[
\begin{cases}
w_t+H(Dw,x)=0
    \qquad& \text{in } \mathbb T^n\times(0,\infty),\\
w(x,0)=g(x)
    \qquad& \text{on } \mathbb T^n.
\end{cases}
\]
Owing to \cite{Fathi08}, for any $u\in C(\T^n)$,  the following limit exists  
\begin{equation}\label{eq:backward}
u_-=\lim_{t\to\infty}\bigl(T_t^-u+ct\bigr)
\qquad\text{uniformly on }\mathbb T^n
\end{equation}
and $u_-$ is a weak KAM solution associated with $u$. 
The following is a basic fact (see \cite{LMT} for instance). 
For the reader's convenience, we present the proof here. 
\begin{lemma}\label{lem:onesided}
Suppose that \(u\) is a critical subsolution. Then
\[
u\leq u_- \qquad \text{on } \mathbb T^n.
\]
Moreover, for fixed \(x_0\in\mathbb T^n\),
\[
u(x_0)=u_-(x_0)
\]
if and only if there exists a backward calibrated curve
\[
\gamma:(-\infty,0]\to\mathbb R^n,
\qquad \gamma(0)=x_0,
\]
such that
\begin{equation}\label{eq:calibrated2}
u(\gamma(t_2))-u(\gamma(t_1))
=
\int_{t_1}^{t_2}
\bigl(L(\dot\gamma(t),\gamma(t))+c\bigr)\,dt \quad \text{for all \(t_1\leq t_2\leq0\).}
\end{equation}

\end{lemma}

\begin{proof}
Since \(u\) is a critical subsolution, by (\ref{eq:dominance}), 
\[
u\leq T_t^-u+ct
\qquad\text{for all }t\geq0.
\]
Letting \(t\to\infty\) gives
\[
u\leq u_-.
\]

Next we assume that \(u(x_0)=u_-(x_0)\). Since \(u_-\) is a backward weak KAM solution, there exists a backward calibrated curve
\[
\gamma:(-\infty,0]\to\mathbb R^n,
\qquad \gamma(0)=x_0,
\]
such that, for every \(s\leq0\),
\[
u_-(x_0)-u_-(\gamma(s))
=
\int_s^0
\bigl(L(\dot\gamma(t),\gamma(t))+c\bigr)\,dt.
\]
Using \(u(x_0)=u_-(x_0)\) and \(u\leq u_-\), we obtain
\[
u(x_0)-u(\gamma(s))
\geq
\int_s^0
\bigl(L(\dot\gamma(t),\gamma(t))+c\bigr)\,dt.
\]
On the other hand, since \(u\) is a critical subsolution, it is dominated by
\(L+c\) (i.e., (\ref{eq:dominance})), and therefore
\[
u(x_0)-u(\gamma(s))
\leq
\int_s^0
\bigl(L(\dot\gamma(t),\gamma(t))+c\bigr)\,dt.
\]
Hence equality holds:
\[
u(x_0)-u(\gamma(s))
=
\int_s^0
\bigl(L(\dot\gamma(t),\gamma(t))+c\bigr)\,dt
\qquad\text{for all }s\leq0.
\]
Subtracting the identities corresponding to \(s=t_1\) and \(s=t_2\) gives
\[
u(\gamma(t_2))-u(\gamma(t_1))
=
\int_{t_1}^{t_2}
\bigl(L(\dot\gamma(t),\gamma(t))+c\bigr)\,dt
\]
for all \(t_1\leq t_2\leq0\).

Now we assume (\ref{eq:calibrated2}). Then due to the definition of $T_t^{-}$, 
\[
u(x_0)\geq T_t^{-}u(x_0)+ct \quad \text{for all $t\geq 0$}.
\]
Sending $t\to \infty$, we derive that $u(x_0)\geq u_-(x_0)$. Hence the equality holds.
\end{proof}

\subsection{Aubry and Mather sets} 
A central goal in dynamical systems is to understand the long-time behavior of trajectories, or at least of distinguished classes of trajectories.
Invariant sets naturally play an important role in this question. 
In KAM theory, the dynamics on a regular KAM torus is particularly rigid: after a change of coordinates, it is conjugate to a linear quasiperiodic flow, so the long-time behavior of every trajectory on the torus is completely described.

The Mather set provides a weaker but robust analogue for action-minimizing dynamics far beyond the perturbative KAM regime. 
By Mather's graph theorem, the lifted Mather set lies on an invariant Lipschitz graph. 
The dynamics on this set need not be conjugate to a linear flow and, in general, the long-time behavior of individual trajectories is not explicitly predictable. 
Nevertheless, the Mather set forms the recurrent core of the globally action-minimizing dynamics: it is the union of the supports of minimizing invariant measures and therefore gives a statistical description of the long-time behavior of minimizing trajectories. 
In dimension two, additional topological and ordering properties yield much stronger control of individual minimizing curves, as in the classical Aubry--Mather theory (see \cite{Bangert88} for instance).

The Aubry set is in general larger than the Mather set. 
Roughly speaking, besides the recurrent minimizing dynamics contained in the Mather set, the Aubry set may also contain globally minimizing static orbits connecting different recurrent components. 
Thus, the Mather set describes the recurrent core of the minimizing dynamics, while the Aubry set captures a larger global geometric structure of action-minimizing trajectories.

\subsubsection{Aubry set}
For \(t>0\), let
\[
    h_t(x,y)
    :=
    \inf_{\substack{\gamma(0)=x\\ \gamma(t)=y}}
    \int_0^t L(\dot\gamma(s),\gamma(s))\,ds,
\]
where the infimum is over absolutely continuous curves.  
The Peierls barrier
is
\[
    h(x,y)
    :=
    \liminf_{t\to\infty}\bigl(h_t(x,y)+ct\bigr).
\]
In particular, if $u$ is  a critical subsolution, then (\ref{eq:dominance}) implies that for all $x,y\in \R^n$ and $t\geq 0$,
\[
u(y)-u(x)\leq h_t(x,y)+ct.
\]
Consequently, 
\[
u(y)-u(x)\leq h(x,y). 
\]
The projected Aubry set is
\[
    \A:=\{x\in\T^n\,:\,h(x,x)=0\}.
\]
On \(\A\), the Mather semidistance is
\begin{equation}\label{eq:deltaM}
    \delta_M(x,y):=h(x,y)+h(y,x)\geq 0.
\end{equation}
The metric quotient obtained from \((\A,\delta_M)\) will be denoted by \((\AM,\delta_M)\).

We recall three standard facts about the projected Aubry set.

\begin{itemize}
  \item {\bf Property 1 (Graph property).}
Every critical subsolution \(u\) is differentiable at each point of
\(\A\), and
\[
H(Du(x),x)=c
\qquad \text{for all }x\in\A.
\]
Moreover, \(Du\) is Lipschitz continuous on \(\A\), and the gradient on \(\A\) is independent of the choice of critical subsolution. 
More precisely, if \(v_1\) and \(v_2\) are two critical subsolutions, then
\[
Dv_1(x)=Dv_2(x)
\qquad \text{for all }x\in\A.
\]

Consequently, for any critical subsolution \(u\), the lifted Aubry set can be represented as
\[
\widetilde{\mathcal A}
=
\left\{
\bigl(D_pH(Du(x),x),x\bigr)\,:\,x\in\mathcal A
\right\}.
\]
This representation is independent of the choice of the critical subsolution \(u\). Moreover, \(\widetilde{\mathcal A}\) is invariant under the Euler--Lagrange flow. More precisely, for every \(x\in\mathcal A\), there exists a  \((u,L,c)\)-calibrated curve $\xi$ satisfying that
\[
\xi:\mathbb R\to\mathcal A,
\qquad \xi(0)=x
\]
and
\[
\dot\xi(t)
=
D_pH\bigl(Du(\xi(t)),\xi(t)\bigr)
\qquad\text{for all }t\in\mathbb R.
\]
\medskip

\item \textbf{Property 2 (Uniqueness set).}
Let \(v_1\) and \(v_2\) be two viscosity solutions of the cell problem
\eqref{eq:cell}. If
\[
v_1=v_2 \qquad \text{on } \mathcal A,
\]
then
\[
v_1=v_2 \qquad \text{on } \mathbb T^n.
\]

\medskip

    \item {\bf Property 3 ($C^1$ critical subsolution strict outside $\A$).} It was proved in Fathi--Siconolfi \cite{FathiSiconolfi04} that there exists \(v\in C^1(\T^n)\) such that
\begin{equation}\label{eq:strict}
\begin{cases}
    H(Dv,x)\leq c
    &\qquad\text{for every }x\in\T^n,\\
    H(Dv,x)<c
    &\qquad\text{for every }x\in\T^n\setminus\A.
\end{cases}
\end{equation}
\end{itemize}

We recall the result that supplies the low-dimensional input.

\begin{theorem}[Fathi--Figalli--Rifford \cite{FFR09}]\label{thm:FFR}
Let \(M\) be a smooth manifold and let \(H:T^*M\to\R\) be a Tonelli Hamiltonian.  
The Mather quotient has vanishing one-dimensional Hausdorff measure in either of the following cases:
\begin{enumerate}[label=\textup{(\roman*)}]
    \item \(\dim M=1\) or \(2\) and \(H\in C^2\);
    \item \(\dim M=3\) and \(H\in C^{k,1}\) for some \(k\geq3\).
\end{enumerate}
That is,
\[
    \mathcal H^1(\AM,\delta_M)=0.
\]
\end{theorem}

\subsubsection{Mather measures and the Mather set}

Let \(\mathcal W\) denote the set of Borel probability measures on \(\mathbb R^n\times\mathbb T^n\) that are invariant under the Euler--Lagrange flow. 
A measure \(\mu\in\mathcal W\) is called a \emph{Mather measure} if
\[
\int_{\mathbb R^n\times\mathbb T^n} L(q,x)\,d\mu
=
\min_{\nu\in\mathcal W}
\int_{\mathbb R^n\times\mathbb T^n} L(q,x)\,d\nu
=-c.
\]
We denote by \(\mathfrak M\) the collection of all Mather measures.

The \emph{Mather set} in phase space is defined by
\[
\widetilde{\mathcal M}
:=
\overline{
\bigcup_{\mu\in\mathfrak M}\operatorname{spt}\mu
}
\subset \mathbb R^n\times\mathbb T^n.
\]
Its projection onto the configuration space,
\[
\mathcal M
:=
\pi_x(\widetilde{\mathcal M})
\subset\mathbb T^n,
\]
is called the \emph{projected Mather set}. 
See \cite{M1, M2}.
It is known that
\[
\widetilde{\mathcal M}\subseteq \widetilde{\A}
\]
We could also recover $\widetilde{\mathcal M}$ from ${\A}$ and a critical subsolution $u$ (equivalently from $\widetilde{\A}$). 
In fact, for any $x\in \A$, let $\xi_x:\R\to \R^n$ be the solution to
\[
\dot \xi_x(t)=D_pH(Du(\xi_x(t)),\xi_x(t)) \qquad \text{for all $t\in \R$}.
\]
Denote by
\[
C_x=\{\text{all measures induced by $\xi_x$}\},
\]
i.e., $\sigma\in C_x$ if there exists a subsequence $T_j\to \infty$ such that for all $f\in C(\R^n\times \T^n)$,
\[
\int_{\R^n\times \T^n}f\,d\sigma=\lim_{j\to \infty}\frac{1}{T_j}\int_{0}^{T_j}f(\dot \xi_x(t), \xi_x(t))\,dt.
\]
Then
\begin{equation}\label{eq:AtoM}
\widetilde{\mathcal M}=\overline{\bigcup_{x\in \A,\sigma\in C_x}\mathrm{supp}(\sigma)}.
\end{equation}

\subsection{\texorpdfstring{$L^\infty$}{L-infinity} variational problems}
Next, we recall the definition of absolute minimizers for $H$, which was first introduced by Aronsson in the 1960s \cite{Aronsson65}. 

\begin{definition}
Let $U\subset \mathbb{R}^n$ be open. 
A function $u\in W_{\mathrm{loc}}^{1,\infty}(U)\cap C(U)$ is called an \emph{absolute minimizer} for $H$ in $U$ if, for every open set $V\Subset U$ and every
\[
v\in W^{1,\infty}(V)\cap C(\overline V)
\]
satisfying
\[
v=u \qquad \text{on } \partial V,
\]
we have
\[
\operatorname*{ess\,sup}_{x\in V} H(Du(x),x)
\le
\operatorname*{ess\,sup}_{x\in V} H(Dv(x),x).
\]
\end{definition}

The Aronsson equation \eqref{eq:Aronsson} may be viewed as the Euler--Lagrange equation associated with the corresponding \(L^\infty\) variational problem. 
Under appropriate assumptions, viscosity solutions of the Aronsson equation and absolute minimizers are equivalent; see \cite{BJW01, CrandallWangYu09, Yu06} for precise results on this equivalence.
See also \cite{Yu07} for dynamical properties of periodic absolute minimizers, equivalently, periodic viscosity solutions of the corresponding Aronsson equation.

\section{Proofs}
Recall that for each \(k\in\mathbb N\), \(u_k\) minimizes
\[
    I_k[w]
    =
    \int_{\T^n} e^{kH(Dw,x)}\,dx
\]
among $W^{1,\infty}(\T^n)$ functions satisfying
\[
    \int_{\T^n}w\,dx=0.
\]
Note that, for every \(q\in[1,\infty)\), we have \cite{Evans03}
\[
\sup_{k\in\mathbb N}
\|u_k\|_{W^{1,q}(\mathbb T^n)}<\infty.
\]
Write $u=u_\infty$ as a subsequential limit of $u_k$ as $k\to \infty$. Without loss of generality, let us assume that $\lim_{k\to \infty}u_k=u$ uniformly on $\T^n$.  

We first prove several lemmas.

\subsection{A localized truncation inequality}

The following observation is the basic variational tool of the proof.

\begin{lemma}[Band truncation]\label{lem:truncation}
Let \(v\in W^{1,\infty}(\T^n)\), and let \(a<b\).  Define
\[
    E_k:=\{x\in\T^n\,:\,a<u_k(x)-v(x)<b\}.
\]
Then
\begin{equation}\label{eq:truncation}
    \int_{E_k}e^{kH(Du_k,x)}\,dx
    \leq
    \int_{E_k}e^{kH(Dv,x)}\,dx.
\end{equation}
\end{lemma}

\begin{proof}
Define the Lipschitz truncation
\[
    T_{a,b}(s)
    :=
    \begin{cases}
        s, \qquad& s\leq a,\\
        a, \qquad& a<s<b,\\
        s-(b-a), \qquad& s\geq b.
    \end{cases}
\]
Set
\[
    \widetilde w_k
    :=
    v+T_{a,b}(u_k-v).
\]
The Sobolev chain rule, together with the fact that the gradient of a Sobolev function vanishes almost everywhere on each of its level sets, gives
\[
    D\widetilde w_k
    =
    \begin{cases}
       Dv, \qquad& \text{a.e. on }E_k,\\
       Du_k, \qquad& \text{a.e. on }\T^n\setminus E_k.
    \end{cases}
\]
Subtract the mean of \(\widetilde w_k\) to obtain an admissible normalized competitor \(w_k\).  
Since the functional depends only on the gradient,
\[
    I_k[u_k]\leq I_k[w_k]=I_k[\widetilde w_k].
\]
The two integrands agree almost everywhere outside \(E_k\).  
Canceling those terms gives \eqref{eq:truncation}.
\end{proof}

\subsection{The Mather quotient and differences of critical subsolutions}

The next elementary observation explains precisely why the Mather quotient is the right object.

\begin{lemma}\label{lem:Lipschitzquotient}
Let \(v_1,v_2\) be two critical subsolutions and set
\[
    z:=v_1-v_2.
\]
Then for all \(x,y\in\T^n\),
\begin{equation}\label{eq:zdelta}
    |z(y)-z(x)|
    \leq \delta_M(x,y).
\end{equation}
Consequently, \(z|_\A\) descends to a \(1\)-Lipschitz function
\[
    \overline z:(\AM,\delta_M)\to\R.
\]
If
\[
    \mathcal H^1(\AM,\delta_M)=0,
\]
then
\begin{equation}\label{eq:imagezero}
    \mathcal H^1(z(\A))=0.
\end{equation}
In particular, \(z(\A)\) contains no nontrivial interval.
\end{lemma}

\begin{proof}
Since for all $x,y\in\T^n$ and $i=1,2$
\[
    v_i(y)-v_i(x)\leq h(x,y),
\]
we have that
\[
\begin{aligned}
    z(y)-z(x)
    &=
    \bigl(v_1(y)-v_1(x)\bigr)
    +
    \bigl(v_2(x)-v_2(y)\bigr)\\
    &\leq
    h(x,y)+h(y,x)
    =
    \delta_M(x,y).
\end{aligned}
\]
Interchanging \(x\) and \(y\) gives \eqref{eq:zdelta}.  
If \(\delta_M(x,y)=0\), then \(z(x)=z(y)\), so \(z\) descends to the quotient and is \(1\)-Lipschitz there.

Hausdorff measure does not increase under a \(1\)-Lipschitz map.  Therefore
\[
    \mathcal H^1(z(\A))
    \leq
    \mathcal H^1(\AM,\delta_M),
\]
which proves \eqref{eq:imagezero}.
\end{proof}

\subsection{Proof of the main theorem} It was proved in \cite{Yu07} that
\[
D_u=\{u=u_-\}.
\]
So,  it suffices to show that
\[
\{u=u_-\}=\mathcal{A}.
\]
Fix once and for all a strict critical subsolution \(v\) as in
\eqref{eq:strict}.
% Write
% \begin{equation}\label{eq:gdef}
%     g(x):=c-H(Dv(x),x).
% \end{equation}
% Then
% \begin{equation}\label{eq:gproperties}
%     g\geq0,
%     \qquad
%     g^{-1}(0)=\A.
% \end{equation}
We now combine the orbit structure, the Mather quotient, and the band
truncation inequality.

\begin{proof}[Proof of \cref{thm:main}]
Since \(u\) is a critical subsolution, 
\begin{equation}\label{eq:AinDu}
    \A\subset \{u=u_-\}.
\end{equation}
It remains to prove the reverse inclusion.

Suppose by contradiction that
\[
    x_0\in \{u=u_-\}\setminus\A.
\]
Let \(\gamma:(-\infty,0]\to \R^n\) be the one-sided calibrated curve from Lemma \ref{lem:onesided} subject to $\gamma(0)=x_0$.  
Then there is $h<0$ such that $\gamma([h,0])\cap \A=\emptyset$. Note that for \(h\leq t_1<t_2\leq 0\), since $H(Dv,x)<c$ outside $\A$,
\begin{align*}
    u(\gamma(t_2))-u(\gamma(t_1))
    &=
    \int_{t_1}^{t_2}
    \bigl(L(\dot\gamma(t),\gamma(t))+c\bigr)\,dt,\\
    v(\gamma(t_2))-v(\gamma(t_1))
    &<
    \int_{t_1}^{t_2}
    \bigl(L(\dot\gamma(t),\gamma(t))+c\bigr)\,dt.
\end{align*}
Hence
\[
u(\gamma(t_2))-v(\gamma(t_2))
>
u(\gamma(t_1))-v(\gamma(t_1)),
\]
i.e., \(z=u-v\) is strictly increasing along \(\gamma\) on \([h,0]\).
Similarly, \(z\) is nondecreasing along \(\gamma\) on \((-\infty,0]\).
Consequently,  $z(\gamma((-\infty,0]))$ contains a nonempty interval:
\[
(\ell_-,\ell_+)\subset z(\gamma(-\infty,0])).
\]
Note  we actually have  $\gamma((-\infty,0])\cap \A=\emptyset$ due to the flow invariance of $\tilde \A$. 
But the local version $\gamma([h,0])\cap \A=\emptyset$ is sufficient for our purpose. 

By \cref{lem:Lipschitzquotient} and the hypothesis
\(\mathcal H^1(\AM,\delta_M)=0\),
\[
    \mathcal H^1(z(\A))=0.
\]
Therefore \(z(\A)\) cannot contain the interval
\((\ell_-,\ell_+)\).  Choose
\begin{equation}\label{eq:chooseS}
    s\in(\ell_-,\ell_+)\setminus z(\A).
\end{equation}
Since \(z(\A)\) is compact, we may choose numbers \(a<s<b\) and
\(\rho>0\) such that
\begin{equation}\label{eq:expandedband}
    [a-\rho,b+\rho]\cap z(\A)=\varnothing.
\end{equation}

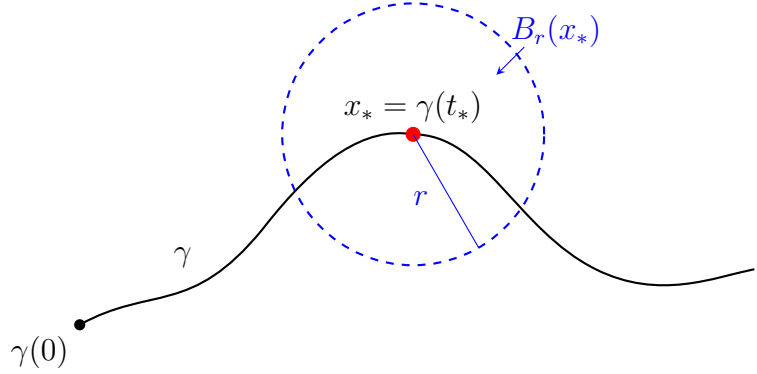
\begin{figure}[ht]
\centering
\begin{tikzpicture}[scale=1.05, >=stealth]

% Center of the ball
\coordinate (xstar) at (4.5,2.7);

% Ball B_r(x_*)
\draw[blue, dashed, thick] (xstar) circle (1.65);

% Curved trajectory gamma
\draw[thick]
(0.3,0.3)
.. controls (1.2,0.8) and (1.7,0.4) ..
(2.6,1.5)
.. controls (3.3,2.4) and (3.9,2.8) ..
(xstar)
.. controls (5.4,2.7) and (5.8,1.5) ..
(6.8,1.0)
.. controls (7.6,0.6) and (8.3,0.9) ..
(8.8,1.0);

% Initial point
\fill (0.3,0.3) circle (2pt);
\node[below left] at (0.3,0.3) {$\gamma(0)$};

% x_* point
\filldraw[red] (xstar) circle (2.5pt);
\node[above] at (xstar) {$x_*=\gamma(t_*)$};

% Label gamma
\node at (1.6,1.15) {$\gamma$};

% Label the ball
\node[blue] at (6.3,4.0) {$B_r(x_*)$};
\draw[blue, ->] (5.9,3.8) -- (5.55,3.45);

% Optional radius
\draw[blue, thin] (xstar) -- ++(-60:1.65);
\node[blue] at (4.6,1.9) {$r$};

\end{tikzpicture}
\caption{The curve $\gamma$ passing through $x_*=\gamma(t_*)$ and the ball
$B_r(x_*)$.}
\label{fig:gamma-ball}
\end{figure}

Choose \(t_*<0\) such that
\[
    z(\gamma(t_*))=s.
\]
Set
\[
    x_*:=\gamma(t_*).
\]
Because \(s\notin z(\A)\), we have \(x_*\notin\A\).  By \eqref{eq:expandedband}, 
\[
\A\cap K_\rho=\emptyset  \qquad \text{for $K_\rho :=
    \bigl\{
       x\in\T^n:
       a-\rho\leq z(x)\leq b+\rho
    \bigr\}$}.
\]
Since \(v\) is strict outside \(\A\), compactness gives \(\eta>0\) such that
\begin{equation}\label{eq:uniformstrict}
    H(Dv(x),x)\leq c-\eta
    \qquad \text{for }x\in K_\rho.
\end{equation}

By continuity of \(z\), choose \(r>0\) such that
\begin{equation}\label{eq:ballband}
    a<z(x)<b
    \qquad \text{for }x\in \overline{B}_r(x_*).
\end{equation}
Let
\[
    z_k:=u_k-v,
    \qquad
    E_k:=\{x\in\T^n:a<z_k(x)<b\}.
\]
Uniform convergence \(u_k\to u\) implies that, for all sufficiently large
\(k\),
\begin{equation}\label{eq:Ekcontain}
    B_r(x_*)\subset E_k\subset K_\rho.
\end{equation}

Applying \cref{lem:truncation}, and  using \eqref{eq:uniformstrict} and \eqref{eq:Ekcontain}, we deduce 
\begin{align}
    \int_{B_r(x_*)}e^{kH(Du_k,x)}\,dx
    &\leq
    \int_{E_k}e^{kH(Du_k,x)}\,dx\notag\\
    &\leq
    \int_{E_k}e^{kH(Dv,x)}\,dx\notag\\
    &\leq
    |E_k| e^{k(c-\eta)}\le e^{k(c-\eta)}.
    \label{eq:strictk}
\end{align}

Fix \(m\geq1\).  
For \(k>m\), H\"older's inequality gives
\begin{align*}
    \int_{B_r(x_*)}e^{mH(Du_k,x)}\,dx
    &\leq
    |B_r(x_*)|^{1-m/k}
    \left(
       \int_{B_r(x_*)}e^{kH(Du_k,x)}\,dx
    \right)^{m/k}\notag\\
    &\leq
    |B_r(x_*)|^{1-m/k} e^{m(c-\eta)}.
    \label{eq:strictm}
\end{align*}
For fixed \(m\), as $p\mapsto e^{m H(p,x)}$ is convex and $\{Du_k\}$ converges weak-* to $Du$, the weak lower semicontinuity (see \cite[Chapter 8]{EvansPDE} for instance) gives
\[
    \int_{B_r(x_*)}e^{mH(Du,x)}\,dx
    \leq
    |B_r(x_*)|e^{m(c-\eta)}.
\]
Letting \(m\to\infty\),
\begin{equation}\label{eq:localstrictlimit}
    \esssup_{B_r(x_*)}H(Du,x)
    \leq
    c-\eta.
\end{equation}
Choose \(\delta>0\) sufficiently small so that $t_*+\delta\leq 0$ and 
\[
\gamma([t_*-\delta,t_*+\delta])
\subset B_r(x_*).
\]
Since $H$ is convex in $p$, \eqref{eq:localstrictlimit} implies that $u$ is a viscosity subsolution of
\[
H(Du,x)\le c-\eta
\qquad\text{in }B_r(x_*).
\]
Fenchel's inequality (\ref{eq:Fenchel}) yields, for
\(t_*-\delta\le t_1<t_2\le t_*+\delta\),
\[
u(\gamma(t_2))-u(\gamma(t_1))
\le
\int_{t_1}^{t_2}
\bigl(L(\dot\gamma(t),\gamma(t))+c-\eta\bigr)\,dt.
\]
See also \cite[Chapter 2]{TranHJ}.
On the other hand, since \(\gamma\) is \((u,L,c)\)-calibrated, 
\[
u(\gamma(t_2))-u(\gamma(t_1))
=
\int_{t_1}^{t_2}
\bigl(L(\dot\gamma(t),\gamma(t))+c\bigr)\,dt,
\]
which is a contradiction. 
Therefore, \(\{u=u_-\}\subset\A\). 
Together with \eqref{eq:AinDu}, this proves
\[
\{u=u_-\}=\A.
\]
Finally, it follows from Theorem 3.4 in \cite{Yu07} that $u$ is a  critical subsolution that is strict outside  $\A$. 
\end{proof}

\begin{proof}[Proof of \cref{cor:lowdim}]
For \(n=1,2\), \cref{thm:FFR} applies already under \(C^2\) regularity. 
For \(n=3\), it applies under \(C^{k,1}\), \(k\geq3\). 
Since the Hamiltonian in \cref{cor:lowdim} is smooth, \eqref{eq:H1zero} holds in every case \(n\leq3\). 
The result follows from \cref{thm:main}.
\end{proof}

\begin{remark}  The proof of \cref{thm:main} actually uses a slightly weaker condition than \eqref{eq:H1zero}.  
It is enough that
\begin{equation}\label{eq:nointerval}
    (u-v)(\A)
    \quad\text{contains no nontrivial interval}.
\end{equation}
The Hausdorff-measure condition is a convenient geometric hypothesis that implies \eqref{eq:nointerval} uniformly for all pairs of critical subsolutions.

This also indicates the main obstruction to extending the argument to higher
dimensions. A non-Aubry calibrated orbit produces an interval of values of
\(u-v\), and the truncation argument works as soon as this interval contains
a value not attained on \(\mathcal A\). Thus the argument can fail only if
the image of the Mather quotient under the induced Lipschitz function is
large enough to contain an interval.

\end{remark}

\section{Open problems}

Our results  leave several natural questions concerning the selection mechanism of the exponential approximation. Although periodic solutions to the Aronsson equation (\ref{eq:Aronsson}) are in general not unique (even up to a constant), the function \(u_k\) is uniquely
determined, after normalization, for each finite \(k\).
Below is a simple example illustrating the selection mechanism.

\begin{example}
Consider \(n=1\) and
\[
H(p,x)=|p|^2+\sin^2 x.
\]
Then
\[
w_1\equiv 0,\qquad w_2=\sin x,\qquad w_3=-\sin x
\]
are all \(2\pi\)-periodic solutions of
\[
A_H(u)=0
\qquad\text{on } 2\pi \T^1
\]
subject to $\int_{2\pi \T^1}u\,dx=0$. However, among these solutions, only \(w_1\) is selected by Evans'
variational principle, since the corresponding minimizer satisfies
\[
u_k\equiv 0
\qquad\text{for every }k\geq 1.
\]
For this case, $c=1=\max_{2\pi\T^1}\sin^2x$ and $\A=\{\frac{\pi}{2}, \frac{3\pi}{2}\}$. Note that 
\[
D_{w_1}=\A \quad \mathrm{and} \quad   D_{w_2}=D_{w_3}=2\pi \T^1 \supsetneq \A. 
\]
\end{example}

 It is therefore
natural to ask whether, in general,  this approximation selects a unique absolute minimizer
as \(k\to\infty\). This question is closely related in spirit to selection problems for
vanishing-viscosity approximations of Hamilton--Jacobi equations. In
\cite{LiuTranYu26}, we showed that, even though the viscous ergodic problem
has a unique normalized solution for every \(\varepsilon>0\), the corresponding
family of solutions need not converge as \(\varepsilon\to0\). It is therefore
particularly interesting to determine whether the exponential variational
approximation considered here has a stronger selection property.

\begin{openproblem}[Uniqueness of the limit]
\label{op:uniqueness}
Assume
\[
\int_{\mathbb T^n} u_k\,dx=0.
\]
Does the whole sequence \(u_k\) converge uniformly as \(k\to\infty\)?
Equivalently, can two different subsequences converge to different absolute
minimizers? By \cref{cor:lowdim}, when $n\leq 3$, every subsequential limit \(u\) satisfies
\[
D_u=\mathcal A,
\]
but this characterization alone does not determine \(u\) uniquely.
\end{openproblem}

\begin{openproblem}[Convergence rate]
\label{op:rate}
If the limit \(u\) in Open Problem~\ref{op:uniqueness} is unique, determine the convergence rate of 
\[
\|u_k-u\|_{L^\infty(\mathbb T^n)}.
\]
In particular, it would be interesting to identify assumptions under which
\[
\|u_k-u\|_\infty\le r(k),
\qquad r(k)\to0,
\]
and to determine the optimal rate \(r(k)\). This is closely related to numerical approximation of the Aubry set.

\end{openproblem}

\begin{openproblem}[Selection of limiting Mather measures and possible recovery of $\M$]
Conceptually, \eqref{eq:AtoM} shows that the Mather set can, in principle, be recovered from the projected Aubry set $\A$ together with the limiting function $u_\infty$. However, this characterization is not particularly convenient from a numerical point of view. It is therefore natural to ask whether Evans' variational approximation also provides a more direct way to recover the projected Mather set $\M$, or equivalently the full Mather set $\widetilde{\M}$.

\label{op:mather-measures}
Let
\[
\sigma_k(x)
:=
\frac{e^{kH(Du_k(x),x)}}
{\displaystyle\int_{\mathbb T^n}
e^{kH(Du_k(y),y)}\,dy}.
\]
The pair \((u_k,\sigma_k)\) is also the logarithmically coupled stationary mean-field game associated with Evans' variational approximation; see \cite{GY20}.
Thanks to \cite{Evans03, Evans04}, along a subsequence,
\[
\sigma_{k_j}\rightharpoonup \sigma
\qquad\text{weakly in the sense of measures},
\]
where \(\sigma\) is a projected Mather measure. In dimension one, under a
suitable nondegeneracy condition, \(\operatorname{spt}(\sigma)\) coincides
with the whole projected Mather set (i.e., $\operatorname{spt}(\sigma)=\mathcal M$),  whereas degeneracy may lead to further
selection (\cite{GISMY10}). What is the corresponding picture in higher dimensions? In
particular, under what conditions does
\[
\operatorname{spt}(\sigma)=\mathcal M,
\]
and what determines the subset of \(\mathcal M\) selected when this equality
fails?
\end{openproblem}

\begin{openproblem}[Higher dimensions]
\label{op:higher-dim}
It remains open whether the identification
\[
D_u=\mathcal A
\]
continues to hold in dimensions \(n\geq4\). The present argument relies on
low-dimensional information on the Mather quotient that is no longer
available in higher dimensions. Is this merely a limitation of the method,
or can genuinely new phenomena occur for \(n\geq4\)? In particular, can one
construct a Hamiltonian and a subsequential limit \(u\) for which
\[
D_u\supsetneq\mathcal A?
\]
\end{openproblem}

\section*{Acknowledgments}

The authors acknowledge the use of ChatGPT 5.6 Plus for suggesting some ideas and assisting with verifications. 
We used some of the suggestions together with our ideas.
All mathematical results and the final writing and revision of the manuscript are due to the authors.

\end{document}